\documentclass[12pt,a4paper]{article}
\usepackage{amsthm}
\usepackage{amsfonts}
\usepackage{mathrsfs}
\usepackage{amscd,amssymb,amsmath,graphicx,verbatim,xy,bm}
\usepackage[pdftex]{hyperref}
\usepackage[TS1,OT1,T1]{fontenc}
\usepackage{geometry}
\newtheorem{theorem}{Theorem}[section]
\newtheorem{lemma}[theorem]{Lemma}
\newtheorem{corollary}[theorem]{Corollary}
\newtheorem{proposition}[theorem]{Proposition}
\newtheorem{notation}[theorem]{}

\theoremstyle{definition}
\newtheorem{definition}[theorem]{Definition}

\theoremstyle{remark}
\newtheorem{remark}[theorem]{Remark}

\catcode`@=11
\let\@int\int \def\int{\displaystyle\@int}
\let\@lim\lim \def\lim{\displaystyle\@lim}
\let\@sup\sup \def\sup{\displaystyle\@sup}
\let\@inf\inf \def\inf{\displaystyle\@inf}
\let\@cap\cap \def\cap{\displaystyle\@cap}
\let\@cup\cup \def\cup{\displaystyle\@cup}
\let\@max\max \def\max{\displaystyle\@max}
\let\@min\min \def\min{\displaystyle\@min}
\let\@iint\iint \def\iint{\displaystyle\@iint}

\def\epsilon{\varepsilon}

\newcommand{\C}{\mathrm{C}}

\begin{document}
\title{ A decomposition theorem for real rank zero inductive limits of 1-dimensional non-commutative CW complexes
\author{Zhichao Liu
\\
{\small\textit{}}}
\date{}
}
\maketitle

\begin{abstract}
  In this paper, we consider the real rank zero $\C^*$-algebras which can be written as an inductive limit of the Elliott-Thomsen building blocks
  and prove a decomposition result for the connecting homomorphisms; this technique will be used in the classification theorem.

  \\
   {\sl 2000 MR Subject Classification:} Primary 46L80; Secondary 19K14, 19K33, 19K35.
\end{abstract}

\section*{Introduction}

   The problem people we considered is to classify $\C^*$-algebras by their K-theoretical data. The first result of this kind was
   the classification of inductive limits of sequence of finite direct sums of matrix algebras
 (called AF algebras)\cite{Ell1}. This result was extended in \cite{Ell2}. It replaced the matrix algebras by matrix algebras over the unit circle and restricted the limit of $\C^*$-algebras to have real rank zero. In \cite{EG}, \cite{EGS} and \cite{DG}, a classification was given to classes of separable nuclear $\C^*$-algebras of real rank zero.

 The decomposition theory plays an important role in the classification theorem. It will have some similar forms and reflect the
 property of the connecting homomorphisms. If $A$ is an $AH$ algebra with real rank zero, then \cite{EG} proves a decomposition result which says
 that $\phi_{m,n}$ can be approximately decomposed as a sum of two parts, $\phi_1$ and $\phi_2$; one part having a very small support
 projection, and the other part factoring through a finite dimensional algebra.

 In this article, the basic building blocks we consider were introduced by Elliott in \cite{Ell3} and Thomsen in \cite{ET}, which are also called one dimensional non-commutative finite CW complexes. The $\C^*$-algebras we consider are the algebras which can be expressed as the real rank zero inductive limit of a sequence
 $$
 A_1\,\xrightarrow{\phi_{1,2}}\,A_2\,\xrightarrow{\phi_{2,3}}\,A_3\rightarrow\cdots
 $$
 with $A_n=\oplus_{i=1}^{n_i}A_{[n,i]}$, where all the $A_{[n,i]}$ are Elliott-Thomsen building blocks and $\phi_{n,n+1}$ are homomorphisms. It will be shown that many inductive limit algebras are within the above class. We  will prove a decomposition result for these algebras, which will be
 used in the proof of uniqueness theorem and of the existence theorem.

 This paper is organized as follows. In Section 1, we will introduce some notations, and collect some known results. In Section 2, we use the generalized
 pairing lemma to get some comparable results about the spectra. In Section 3, we will prove the main decomposition result.

%%%%%%%%%%%%%%%%%%%%%%%%%%%%%%%%%%%%%%%%%%%%%%%%%%%%%第一章%%%%%%%%%%%%%%%%%%%%%%%%%%%%%%%%%

\section{Preliminaries}
\begin{definition}
  Let $F_1$ and $F_2$ be two finite dimensional $\C^*$-algebras. Suppose that there are two unital homomorphisms $\varphi_0,\varphi_1\,:F_1\rightarrow F_2$. Denote
  $$
  A=A(F_1,F_2,\varphi_0,\varphi_1)=\{(f,a)\in C([0,1],F_2)\oplus F_1 :f(0)=\varphi_0(a)\,\,and\,\,f(1)=\varphi_1(a)\}.
  $$
\end{definition}
 These $\C^*$-algebras have been introduced into the Elliott program by Elliott and Thomsen. Denote by $\mathcal{C}$ the class of all unital $\C^*$-algebras of the form $A(F_1,F_2,\varphi_0,\varphi_1)$ and all the finite dimensional $\C^*$-algebras. These $\C^*$-algebras will be called Elliott-Thomsen building blocks. A unital ${\mathrm C}^*$-algebra $A\in\mathcal{C}$ is minimal, or a minimal block, if it is indecomposable, i.e., not the direct sum of two or more ${\mathrm C}^*$-algebras in $\mathcal{C}$(see also \cite{GLN}).

\begin{proposition}
  Let $A=A(F_1,F_2,\varphi_0,\varphi_1)$, where $F_1=\oplus_{j=1}^p M_{k_j}(\mathbb{C}),\,F_2=\oplus_{i=1}^lM_{l_i}(\mathbb{C})$ and $\varphi_0,\varphi_1 :F_1\rightarrow F_2$ be unital homomorphisms.
  Let $\varphi_{0*},\varphi_{1*}:K_0(F_1)=\mathbb{Z}^p\rightarrow K_0(F_1)=\mathbb{Z}^l$ be represented by matrices $\alpha=({\alpha_{ij}})_{l\times p}$ and $\beta=({\beta_{ij}})_{l\times p}$, where $\alpha_{ij},\beta_{ij}\in \mathbb{Z}_+$ for each pair $i,j$. Then
  $$
  K_0(A)=
  \left\{
  \left(
\begin{array}{c}
  x_1 \\
  x_2 \\
   \vdots\\
  x_p
\end{array}
  \right)
\in \mathbb{Z}^p,\,\,with\,\,(\alpha -\beta)
  \left(
\begin{array}{c}
  x_1 \\
  x_2 \\
   \vdots\\
  x_p
\end{array}
  \right)
  =0
\right\}
 $$
 $$
 K_1(A)=\mathbb{Z}^l/Im(\alpha-\beta).
 $$
\end{proposition}

\begin{notation}\rm
   We use the notation $\#(\cdot)$ to denote the cardinal number of the set, the sets under consideration will be sets with multiplicity,
  and then we  shall also count multiplicity when we use the notation $\#$.
\end{notation}
 \begin{notation}
 {\rm
 Let $\phi:A\rightarrow M_n(\mathbb{C})$ be a homomorphism. Then there exists a unitary $u$ such that
 $$
 \phi(f,a)=u^*\cdot{\rm diag}\big(\underbrace{a(\theta_1),
 \cdots,a(\theta_1)}_{t_1},\cdots,\underbrace{a(\theta_p),\cdots,a(\theta_p)}_{t_p},f(y_1),\cdots,f(y_{\bullet})\big)\cdot u,
 $$
 where $y_1,y_2,\cdots,y_{\bullet}\in \coprod_{i=1}^{l}[0,1]_i$. We shall use $\bullet$ to denote any possible positive integer. For $y=(0,i)$(also denoted by $0_i$), one can replace $f(y)$ by
 $$\big(\underbrace{a(\theta_1),
 \cdots,a(\theta_1)}_{\alpha_{i1}},\cdots,\underbrace{a(\theta_p),\cdots,a(\theta_p)}_{\alpha_{ip}}\big)$$
 in the above expression, and do the same with $y=(1,i)$. After this procedure, we can assume each $y_k$ is strictly in the open interval $(0,1)_i$ for some $i$.
 We write the spectrum of $\phi$ by
 $$
 Sp\phi=\{\theta_1^{\thicksim t_1},\theta_2^{\thicksim t_2},\cdots,\theta_p^{\thicksim t_p},y_1,y_2,\cdots,y_{\bullet}\},
 $$
 where $y_k\in \coprod_{i=1}^{l}(0,1)_i$ and we use $\{{\theta_j}^{\thicksim t_j}\}$ to denote $\{\underbrace{\theta_j,\cdots,\theta_j}_{t_j\ times}\}$.

Let $\omega_i=\#(Sp\phi\cap(0,1)_i)$ denote the number of $y_k$'s which are in the $i^{th}$ open interval $(0,1)_i$  counting multiplicity.
  If $\phi$ is unital then we have
 $$
 \sum_{j=1}^{p}t_jk_j+\sum_{i=1}^{l}\omega_il_i=n.
 $$

 If $f=f^*\in A$, we use $Eg(\phi(f))$ to denote the eigenvalue list of $\phi(f)$, then

 $$\#(Eg(\phi(f)))=n\,\,{\rm(counting\,\,multiplicity)}.$$

 }
 \end{notation}

\begin{notation}
  {\rm
  Let us use $\theta_1,\theta_2,\cdots,\theta_p$ denote the spectrum of $F_1$ and denote the spectrum of $C([0,1],F_2)$ by $(t,i)$,
 where $0\leq t\leq1$ and $i\in \{1,2,\cdots,l\}$ indicates that it is in $i$th block of $F_2$. So
 $$
 Sp(C([0,1],F_2))=\coprod_{i=1}^{l}\{(t,i),\,0\leq t\leq1\}.
 $$
   Using identification of $f(0)=\varphi_0(a)$ and $f(1)=\varphi_1(a)$ for $(f,a)\in A,\,(0,i)\in Sp(C[0,1])$ is identified with
 $$
 (\theta_1^{\thicksim\alpha_{i1}},\theta_2^{\thicksim\alpha_{i2}},\cdots,\theta_p^{\thicksim\alpha_{ip}})\subset Sp(F_1)
 $$
 and $(1,i)\in Sp(C([0,1],F_2))$ is identified with
 $$
 (\theta_1^{\thicksim\beta_{i1}},\theta_2^{\thicksim\beta_{i2}},\cdots,\theta_p^{\thicksim\beta_{ip}})\subset Sp(F_1)
 $$
 as in $Sp(A)=Sp(F_1)\cup \coprod_{i=1}^{l}(0,1)_i$.
 }
\end{notation}

\begin{notation}
  {\rm
  Let $A=A(F_1,F_2,\varphi_0,\varphi_1)$. Written $a\in F_1$ as $a=(a(\theta_1),a(\theta_2),\cdots,a(\theta_p))$,
  $f(t)\in C([0,1],F_2)$ as $$f(t)=(f(t,1),f(t,2),\cdots,f(t,l))$$ where $a(\theta_j)\in M_{k_j}(\mathbb{C})$,
  $f(t,i)\in C([0,1],M_{l_i}(\mathbb{C}))$. For any $(f,a)\in A$ and $i\in\{1,2,\cdots,l\}$,
  define $\pi_t:A\rightarrow C([0,1],F_2)$ by $\pi_t(f,a)=f(t)$ and $\pi_t^i:A\rightarrow C([0,1],M_{l_i}(\mathbb{C}))$ by
  $\pi_t^i(f,a)=f(t,i)$ where $t\in(0,1)$ and $\pi_0^i(f,a)=f(0,i)$ (denoted by $\varphi_0^i(a)$), $\pi_1^i(f,a)=f(1,i)$
  (denoted by $\varphi_1^i(a)$). There is a canonical map $\pi_e:\,A\rightarrow F_1$ defined by $\pi_e((f,a))=a$,
  for all $j=\{1,2,\cdots,p\}$, define $\pi_e^j:\,A\rightarrow M_{k_j}(\mathbb{C})$ by $\pi_e^j((f,a))=a(\theta_j)$.

  Define $\widetilde{\pi}_0^j: \varphi_0(F_1)\rightarrow M_{k_j}(\mathbb{C})$ by
  $$
  \widetilde{\pi}_0^j\circ\varphi_0(a)={\rm sgn}(\sum_{i=1}^{l}\alpha_{ij})\cdot a(\theta_j)
  $$
  and  $\widetilde{\pi}_1^j: \varphi_1(F_1)\rightarrow M_{k_j}(\mathbb{C})$ by
  $$
  \widetilde{\pi}_1^j\circ\varphi_1(a)={\rm sgn}(\sum_{i=1}^{l}\beta_{ij})\cdot a(\theta_j)
  $$
  for each $j\in\{1,2,\cdots,p\}$ and ${\rm sgn}(x)$ is the sign function.
  }
\end{notation}

\begin{notation}\rm
  In this paper we use the convention that $A=A(F_1,F_2,\varphi_0,\varphi_1),\,B=B(F_1',F_2',\varphi_0',\varphi_1')$,
  with
  $$
  F_1=\bigoplus_{j=1}^p M_{k_j}(\mathbb{C}),\,\,F_2=\bigoplus_{i=1}^lM_{l_i}(\mathbb{C}),\,\,F_1'=\bigoplus_{j'=1}^{p'}M_{k_{j'}'}(\mathbb{C}),\,\,F_2'=\bigoplus_{i'=1}^{l'} M_{l_{i'}'}(\mathbb{C}).
  $$
  Let $\varphi_{0*},\varphi_{1*}$ be represented by matrices $\alpha=({\alpha_{ij}})_{l\times p}$ and $\beta=({\beta_{ij}})_{l\times p}$, and let $\varphi_{0*}',\varphi_{1*}'$ be represented by matrices $\alpha'=({\alpha_{i'j'}'})_{l'\times p'}$ and $\beta=({\beta'_{i'j'}})_{l'\times p'}$.

  Let $L(A)=\sum_{i=1}^ll_i$, $L(B)=\sum_{i'=1}^{l'}l_{i'}'$. Denote $\{e_{ss'}^i\}(1\leq i\leq l,\,1\leq s,s'\leq l_i)$  the set of matrix units for $\oplus_{i=1}^p M_{l_i}({\mathbb{C}})$ and $\{f_{ss'}^j\}(1\leq i\leq p,\,1\leq s,s'\leq k_j)$ the set of matrix units for $\oplus_{j=1}^l  M_{k_j}({\mathbb{C}})$.
\end{notation}

 % %%%%%%%%%%%%%%%%%%%%%%%%%%%%%%%%%%%%%%%%%%%%%%%%%%%%%%%%%%%%% 测试函数

 \begin{notation}
{\rm
  For each $\eta=\frac{1}{m}$ where $m\in\mathbb{N}_+$. Let $0=x_0<x_1<\cdots<x_{m}=1$ be a partition of $[0,1]$ into $m$ subintervals with equal length $\frac{1}{m}$. We will define a finite subset $H(\eta)\subset A_+$, consisting of two kinds of
 elements as described below.

 (a)For each subset $X_j=\{\theta_j\}\subset Sp(F_1)=\{\theta_1,\theta_2,\cdots,\theta_p\}$ and a list of integers $a_1,b_2,\cdots,a_l,b_l$ with
 $0\leq a_i<a_i+2\leq b_i\leq m$, denote $W_j\triangleq\cup_{\alpha_{ij}\neq0}[0,a_i]_i\cup\cup_{\beta_{ij}\neq 0}[b_i,1]_i$.
 Then we call $W_j$  the closed neighborhood of $X_j$, we define element $(f,a)\in A_+$ corresponding to $X_j$  and $W_j$  as follows:

 For each $t\in[0,1]_i,\,i=\{1,2,\cdots,l\}$, define
 $$
 f(t,i)=
 \begin{cases}
  \varphi_0^i(f_{11}^j)\dfrac{\eta-dist(t,[0,a_i\eta]_i)}{\eta}, & \mbox{if } 0\leq t\leq (a_i+1)\eta \\
  0, & \mbox{if } (a_i+1)\eta\leq t\leq (b_i-1)\eta \\
  \varphi_1^i(f_{11}^j)\dfrac{\eta-dist(t,[b_i\eta,1]_i)}{\eta}, & \mbox{if } (b_i-1)\eta\leq t\leq 1
\end{cases}
$$
  All such elements $(f,a)=(f(t,1),f(t,2),\cdots,f(t,l))\in A_+$ are included in the set $H(\eta)$ and are called test functions of type 1.

(b)For each closed subset $X=\cup_s[x_{r_s},x_{r_{s+1}}]_i \subset [\eta,1-\eta]_i$(the finite union of
closed intervals $[x_r,x_{r+1}]$ and points $\{x_r\}$). So there are finite subsets for each $i$. Define $(f,a)$ corresponding to $X$ by $a=0$ and for each
$t\in(0,1)_r,r\neq i,f(t,r)=0$ and for $t\in(0,1)_i$ define
$$
f(t,i)=
\begin{cases}
  e_{11}^i\big(1-\dfrac{dist(t,X)}{\eta}\big )e_{11}^i, & \mbox{if } dist(t,X)<\eta \\
  0, & \mbox{if } dist(t,X)\geq \eta.
\end{cases}
$$
here we use $e_{11}^i$ as the matrix units of $M_{l_i}(\mathbb{C})$. All such elements are called test functions of type 2.

Note that for any closed subset $Y\subset[\eta,1-\eta]$, there is a closed subset $X$ consisting of the union of the
intervals $[x_r,x_{r+1}]$ such that $X\supset Y$ and for any $x\in X$, $dist(x,Y)\leq\eta$.

}
 \end{notation}

 \begin{notation}
 {\rm
 Let $\phi:A\rightarrow M_n(\mathbb{C})$ be a homomorphism. $\phi$ has the expression as 1.3. For a closed neighborhood $W_j$ of $X_j=\{\theta_j\}$, define
 $$
 \#_{X_j}\big (Sp\phi\cap W_j \big):=
 t_j+\sum_{i=1}^{l}\alpha_{ij}\#\big (Sp\phi\cap [0,a_i\eta]_i\big)+\sum_{i=1}^{l}\beta_{ij}\#\big (Sp\phi\cap [b_i\eta,1]_i \big).
 $$
 }
 \end{notation}

 \begin{notation}
 {\rm
 In general, a unital homomorphism $\phi: C[0,1]\rightarrow B$ with finite dimensional image is always of the form :
 $$
 \phi(f)=\sum f(x_k)p_k,\,\,\forall f\in C[0,1]
 $$
 where $\{x_k\}$ is a finite subset of [0,1] and $\{p_k\}$ is a set of mutually orthogonal projections with $\sum p_k=1_B$.
 A homomorphism $\phi: A=M_n(C([0,1]))\rightarrow B$ with finite dimensional image is of the form:
 $$
 \phi(f)=\sum f(x_k)\otimes p_k,\,\,\forall f\in A
 $$
 for a certain identification of $\phi(1_A)B\phi(1_A)\cong M_n(\mathbb{C})\otimes (\phi(e_{11})B\phi(e_{11}))$,
  where $\{p_k\}$ is a set of mutually orthogonal projections in $\phi(e_{11})B\phi(e_{11})$.

 }
 \end{notation}

\section{Pairing results}

\begin{definition}
  A unital $\C^*$-algebra is said to have $real\ rank\ zero$, written $RR(A)=0$, if the set of invertible self-adjoint elements is dense in $A_{sa}$.
\end{definition}

\begin{notation}
 {\rm
  Suppose $A$ is a $\C^*$-algebra, $B\subset A$ is a subalgebra, $F\subset A$ is a finite subset and let $\varepsilon>0$. If
 for each $f\in F$, there exists an element $g\in B$ such that $\|f-g\|<\varepsilon$, then we shall say that $F$ is approximately contained in $B$ to within $\varepsilon$, and denote this by $F\subset_{\varepsilon} B$.
 }
\end{notation}
  The following is clear by the standard techniques of spectral theory \cite{BBEK}.
\begin{lemma}
  Let $A=\underrightarrow{lim}(A_n,\phi_{m,n})$ be a $\C^*$-algebra which is the inductive limit of $\C^*$-algebras $A_n$ with morphisms
  $\phi_{m,n}:A_m\rightarrow A_n$. Then $A$ has $RR(A)=0$ if and only if for any finite self-adjoint subset $F\subset A_m$ and $\varepsilon >0$, there exists $n\geq m$ such
  that $$\phi_{m,n}(F)\subset_{\epsilon}\{ f\in (A_n)_{sa}\,|\,f\,has\,finite\,spectrum\}.$$
\end{lemma}

\begin{notation}
  {\rm
  The following fact is well known. (called Pairing Lemma\cite{Su} and Marriage Lemma in \cite{HV})

  Let $X,Y$ be two sets of finitely many points in a metric space, $\#(X)=\#(Y)$. Suppose that for any subset $\widetilde{X}\subset X$,
  $$
  \#\big \{y\in Y,\, dist(y,\widetilde{X})<\varepsilon\big \}\geq\#(\widetilde{X}),
  $$
  and any subset $\widetilde{Y}\subset Y$,
  $$
  \#\big\{x\in X,\, dist(x,\widetilde{Y})<\varepsilon \big \}\geq\#(\widetilde{Y}).
  $$
  Then $X$ and $Y$ can be paired to within $\varepsilon$ one by one.

  }
\end{notation}

  We need the following generalization of the pairing lemma due to Gong  (see \cite{EGLN})

\begin{lemma}\label{2.5}
  Let $X,Y$ be two sets of a metric space with subsets $X_1\subset X$ and $Y_1\subset Y$. Suppose that for each subset $\widetilde{X}\subset X_1$,
  $$
  \#\big \{y\in Y,\, dist(y,\widetilde{X})<\varepsilon\big \}\geq\#(\widetilde{X})
  $$
  and for each subset $\widetilde{Y}\subset Y_1$,
  $$
  \#\big\{x\in X,\, dist(x,\widetilde{Y})<\varepsilon \big \}\geq\#(\widetilde{Y}).
  $$
  Then there are subsets $X_0\subset X$ ,$Y_0\subset Y$ with $X_0\supset X_1$, $Y_0\supset Y_1$ such that $X_0$ and $Y_0$ can be paired to within $\varepsilon$ one by one.
\end{lemma}

\begin{lemma} \label{2.6}
  Let $A\in \mathcal{C}$, for any $1>\varepsilon >0$ and $\eta=\frac{1}{m}$ where $m\in\mathbb{N}_+$, if $\phi,\psi:A\rightarrow M_n(\mathbb{C})$ are unital homomorphisms with the condition that $Eg(\phi(h))$ and $Eg(\varphi(h))$ can be paired to within $\frac{1}{4}$ one by one for all $h\in H(\eta)$, then for each $i\in \{1,2,\cdots,l\}$, then there exists $X_i\subset Sp\phi\cap(0,1)_i$, $Y_i\subset Sp\psi\cap(0,1)_i$
   with $X_i\supset Sp\phi\cap[\eta,1-\eta]_i$ , $Y_i\supset Sp\psi\cap[\eta,1-\eta]_i$ such that $X_i$ and $Y_i$ can be paired to within $2\eta$ one by one.
\end{lemma}
\begin{proof}
  We will apply lemma \ref{2.5} to the set $X\triangleq Sp\phi\cap(0,1)_i$,\ $X_1\triangleq Sp\phi\cap[\eta,1-\eta]_i$ \ and
  $Y\triangleq Sp\psi\cap(0,1)_i$,\ $Y_1\triangleq Sp\psi\cap[\eta,1-\eta]_i$.

  For any subset $\widetilde{X}\subset X_1$, let
  $$
  X'=\bigcup\{\ [r_s\eta,r_{s+1}\eta]_i\,|\, \widetilde{X}\cap[r_s\eta,r_{s+1}\eta]_i \neq \varnothing \}
  $$
  and
  $$
  X''=\{\ x\in [0,1]_i\,|\, dist(x,X')\leq\eta\}.
  $$
  Let $h\in H(\eta)$ corresponding to $X'$, then $h|_{X'}=1$ and $support(h)\subset X''$. Since $Eg(\phi(h))$ and $Eg(\psi(h))$ can be paired within $\varepsilon$, then we will have
  $$
  \#\big \{y\in Y,\,dist(y,\widetilde{X})<2\eta\big \}\geq\#(\widetilde{X}).
  $$
  Similarly for any subset $\widetilde{Y}\subset Y_1$,we also have
  $$
  \#\big \{x\in X,\,dist(x,\widetilde{Y})<2\eta\big \}\geq\#(\widetilde{Y}).
  $$
  Then there exists $X_0\subset X$ ,$Y_0\subset Y$ with $X_0\supset X_1$, $Y_0\supset Y_1$ such that $X_0$ and $Y_0$ can be paired one by one to within $2\eta$.

 \end{proof}
 \begin{lemma} \label{2.7}
   Let $A,B\in \mathcal{C}$, $\phi: A\rightarrow B$ be a unital homomorphism, $V\subset (0,1)_i $ is a closed interval, $W_j$ is a closed neighborhood of $X_j=\{\theta_j\}$, then for any $x_0\in [0,1]$, there exists $\eta,\delta>0$ such that for all $x\in B(x_0,\delta)$,

   (1)$\#\big (Sp\phi_x\cap B(V,2\eta)\big)=\#\big (Sp\phi_{x_0}\cap V\big)$,\,\,$Sp\phi_x\cap \overline{B(V,3\eta)}\backslash B(V,2\eta)=\varnothing$;

   (2)$\#_{X_j}\big(Sp\phi_{x}\cap \overline{B(W_j,2\eta)}\big)=\#_{X_j}\big(Sp\phi_{x_0}\cap W_j\big),\,\,
   Sp\phi_x\cap \overline{B(W_j,3\eta)}\backslash B(W_j,2\eta)=\varnothing$.

 \end{lemma}
 \begin{proof}
  Let $\phi_{x_0}:=\pi_{x_0}\circ \phi $. We will denote $Sp\phi_{x_0}$ by
  $$
  Sp\phi_{x_0}=\{\theta_1^{\sim t_1(x_0)},\theta_2^{\sim t_2(x_0)},\cdots,\theta_p^{\sim t_p(x_0)},y_1^1,y_2^1,\cdots,y_{\bullet}^1,\cdots,y_1^l,y_2^l,\cdots,y_{\bullet}^l\}
  $$
  where $y_1^i,y_2^i,\cdots,y_{\bullet}^i\in (0,1)_i$, for $i=1,2,\cdots,l$.

  Let
  $$
  \eta_1=min\{dist(y_r^i,y_s^i),\,y_r^i\neq y_s^i,\, for\ all\ possible \,r,s \, \},
  $$
  $$
  \eta_2=min\{dist(0_i,V),\,dist(1_i,V),\,dist(y_r^i,V),\,y_r^i\notin V \},
  $$
  $$
  \eta_3=min\{dist(y_r^i,W_j),\,y_r^i\notin W_j \}.
  $$
  There exists a large enough integer $m$ such that
  $\frac{1}{m}<\frac{1}{5}min\{\eta_1,\eta_2,\eta_3\}$, set $\eta=\frac{1}{m}$, then we
  will construct a finite subset $H(\eta)\subset A_+$, by the continuity of $\phi$, there exists $\delta>0$ such that
  $$
  \|\phi_{x}(h)-\phi_{x_0}(h)\|<1,
  $$
  for all $h\in H(\eta)$, $x\in B(x_0,\delta)$.

  Then by the Weyl spectral variation inequality \cite{B}, $Eg(\phi_x(h))$ and $Eg(\phi_{x_0}(h))$ can be paired to within 1 one by one.
  By lemma \ref{2.6},  there exist $X_i(x)\subset Sp\phi_{x}\cap(0,1)_i$, $Y_i(x_0)\subset Sp\phi_{x_0}\cap(0,1)_i$
  with $X_i(x)\supset Sp\phi_{x}\cap[\eta,1-\eta]_i$ , $Y_i(x_0)\supset Sp\phi_{x_0}\cap[\eta,1-\eta]_i$ such that
  $X_i(x)$ and $Y_i(x_0)$ can be paired to within $2\eta$ one by one.

  From the construction of $\eta$, we have $V\subset [5\eta,1-5\eta]_i$, and from the pairing results, we have
  \begin{eqnarray*}
    \#\big (Sp\phi_{x_0}\cap V\big) &\leq& \#\big (Sp\phi_x\cap B(V,2\eta)\big) \\
     &\leq & \#\big(Sp\phi_{x}\cap \overline{B(V,3\eta)}\big) \\
     &=& \#\big(Sp\phi_{x_0}\cap \overline{B(V,5\eta)}\big)
  \end{eqnarray*}
  Similarly,
  \begin{eqnarray*}
    \#_{X_j}\big (Sp\phi_{x_0}\cap W_j\big) &\leq&\#_{X_j}\big (Sp\phi_x\cap B(W_j,2\eta)\big) \\
     &\leq & \#_{X_j}\big(Sp\phi_{x}\cap \overline{B(W_j,3\eta)}\big) \\
     &=& \#_{X_j}\big(Sp\phi_{x_0}\cap \overline{B(W_j,5\eta)}\big)
  \end{eqnarray*}
  holds for any $x\in [0,1]$.

  Since
  $$
  \#\big(Sp\phi_{x_0}\cap V\big)=\#\big(Sp\phi_{x_0}\cap \overline{B(V,5\eta)}\big),
  $$
  $$
  \#_{X_j}\big(Sp\phi_{x_0}\cap W_j\big)=\#_{X_j}\big(Sp\phi_{x_0}\cap \overline{B(W_j,5\eta)}\big).
  $$

  It is obvious that
  $$
  \#\big(Sp\phi_x\cap B(V,2\eta)\big)=\#\big(Sp\phi_{x_0}\cap V\big),\,\,
  Sp\phi_x\cap \overline{B(V,3\eta)}\backslash B(V,2\eta)=\varnothing
  $$
  and
  $$
  \#_{X_j}\big(Sp\phi_{x}\cap \overline{B(W_j,2\eta)}\big)=\#_{X_j}\big(Sp\phi_{x_0}\cap W_j\big),\,\,
   Sp\phi_x\cap \overline{B(W_j,3\eta)}\backslash B(W_j,2\eta)=\varnothing
  $$
  hold for all $x\in B(x_0,\delta)$.
 \end{proof}

 %%%%%%%%%%%%%%%%%%%%%%%%%%%%%%%%%%%%%%%%%%%%%%%

\section{ Decomposition theorem }

    The following lemma is Proposition 3.2 of \cite{DNNP}.
 \begin{lemma} \label{3.1}
  Let $X$ be a Hausdorff compact space ,let $k'> k\geq 1$ be integers ,let $\mathcal{W}$ be an open cover of $X$, and assume that for each $W\in \mathcal{W}$ there is given a continuous projection valued map
  $p_W:W\rightarrow M_n$ such that ${\rm rank}\ p_W\geq k'$ for $x\in W$. If dim$(X)\leq2(k'-k)-1$, then there is a continuous projection valued map $p:X\rightarrow M_n$ such that for $x\in X$:
  $$
  {\rm rank}\ p(x)\geq k,
  $$
  $$
  p(x)\leq\bigvee\{p_W(x):W\in \mathcal{W},x\in W\}.
  $$
\end{lemma}
\begin{notation}\rm
  Let $A=A(F_1,F_2,\varphi_0,\varphi_1)\in\mathcal{C}$ be a minimal block, $p\in A$ be a projection, we will take $p$ as a
  projection in $C([0,1],M_{L(A)}(\mathbb{C}))$, where $L(A)=\sum_{i=1}^{l}l_i$, then  rank $p\in \mathbb{N}$.

  For a projection $p\in A=\oplus_{k=1}^{n}A_k$, where $A_1,A_2,\cdots,A_n\in \mathcal{C}$ are minimal blocks, let
   $\sigma_k:A\rightarrow A_k$ be a homomorphism which maps the direct sum to the $k^{th}$ block in a natural way, define
  $$
  {\rm rank}\,\,p=({\rm rank}\,\,\sigma_1(p),{\rm rank}\,\,\sigma_2(p),\cdots,{\rm rank}\,\,\sigma_n(p))\in \mathbb{N}^n,
  $$
  where  we take $\sigma_k(p)$  as a projection in $C([0,1],M_{L(A_k)}(\mathbb{C}))$. Let $p,q\in A$ be two projections,
  we say that ${\rm rank}\,\,p\geq{\rm rank}\,\,q$, if ${\rm rank}\,\,\sigma_k(p)\geq{\rm rank}\,\,\sigma_k(q)$ holds for all $i\in \{1,2,\cdots,l\}$.
\end{notation}

\begin{theorem} \label{dc}
  Let $A=A(F_1,F_2,\varphi_0,\varphi_1),B=B(F_1',F_2',\varphi_0',\varphi_1')\in \mathcal{C}$, assume that $F_2'$ has only one block. Let
  $G\subset A$ be a finite set, for any positive integers $J,L$, there exists $\eta>0$ such that if a unital homomorphism $\phi: A\rightarrow B$ satisfies $\phi(H(\eta))\subset_{\frac{1}{6}}\{ f\in B\, |f\,has\,finite\,spectrum\}$, then there exists a projection $q\in B$
  and a unital homomorphism $\psi: A\rightarrow (1-q)B(1-q)$ with finite dimensional image such that

  $(1)\,\,L\cdot{\rm rank}\,\,(\phi(e)-\psi(e))<{\rm rank}\,\,\psi(e)$ for any projection $e\in A$,

  $(2)\,\,\|q\phi(g)-\phi(g)q\|<\dfrac{4}{J} \,\,for\, \,any\,\, g\in G$,

  $(3)\,\,\|\phi(g)-q\phi(g)q\oplus\psi(g)\|<\dfrac{4}{J} \,\,for\,\,any\,\, g\in G$.
\end{theorem}

  \begin{proof}
  Since $G\subset A$ is a finite set, $K_0(A)$ is finitely generated, we may assume that $G$ contains a set of projections
  which generate $K_0(A)$. Then there exists an integer $K>0$ such that for any $x,x'\in [0,1]_i$ with $d(x,x')<\frac{2}{K}$,
  $\|g(x)-g(x')\| <\frac{1}{J}$ holds for all $g\in G$.\ Set $\eta=\frac{1}{8K(L+1)}$ and take $H(\eta)$ as test functions.

  Let $A=A(F_1,F_2,\varphi_0,\varphi_1)$, $B=B(F_1',F_2',\varphi_0',\varphi_1')$, where
  $$F_1=\bigoplus_{j=1}^p M_{k_j}(\mathbb{C}),\,\,F_2=\bigoplus_{i=1}^l M_{l_i}(\mathbb{C}),\,\,
  F_1'=\bigoplus_{j'=1}^{p'}M_{k_{j'}'}(\mathbb{C}),\,\,F_2'=M_{l'}(\mathbb{C}).$$
  Let $\varphi_{0*},\varphi_{1*}$ be represented by matrices $\alpha=({\alpha_{ij}})_{l\times p}$ and $\beta=({\beta_{ij}})_{l\times p}$, and let $\varphi_{0*}',\varphi_{1*}'$ be represented by matrices $\alpha'=({\alpha_{j'}'})_{1\times p'}$ and $\beta=({\beta'_{j'}})_{1\times p'}$.

  Since $Sp(B)=Sp(F_1')\cup (0,1)$, we can choose a base point $x_0\in (0,1)$,  Set $\phi_{x_0}=\pi_{x_0}\circ \phi$, denote
  $Sp\phi_{x_0}$ by
  $$
  Sp\phi_{x_0}=\{\theta_1^{\sim t_1},\theta_2^{\sim t_2},\cdots,\theta_p^{\sim t_p},y_1,y_2,\cdots,y_{\bullet}\}
  $$
  where $y_1,y_2,\cdots,y_{\bullet}\in \coprod_{i=1}^{l}(0,1)_i\subset Sp(A)$.

  $Step\,1$:  %Construct subsets in $\coprod_{i=1}^{l}[0,1]_i$.           %第一步构造不同的集合
  We will use the technique in lemma 2.21 of \cite{EG} to construct disjoint closed intervals
  $V_0^i,V_1^i,V_2^i,\cdots,V_K^i$ in $[0,1]_i\subset Sp(A)$ for $i=1,2,\cdots,l$ with the following properties:

  (i) Define closed sets $W_r^i$ by
  $$
  W_r^i=\{x\in [0,1]_i,\quad d(x,V_r^i)\leq2\eta=\frac{1}{4K(L+1)}\}.
  $$
  One has $W_{r_1}^i\cap W_{r_2}^i=\varnothing$ if $r_1\neq r_2$.

  (ii) Diameter $(W_r^i)<\dfrac{2}{K}$ for each $r$.

  (iii) For all $i\in {1,2,\cdots,l}$, we have
   $$
   \sum_{r=0}^{K}\#(Sp\phi_{x_0}\cap V_r^i)\geq\frac{L}{L+1}\#(Sp\phi_{x_0}\cap (0,1)_i).
   $$

   Set $t_0'=0,\ t_1'=\frac{1}{K},\cdots,\,t_K'=\frac{K}{K}=1$. Consider each set $(t_{r-1}',t_r')_i$, we will define intermediate points
   $S_r,\widetilde{S}_r$ with $t_{r-1}'\leq S_r \leq \widetilde{S}_r\leq t_r'$, $(1\leq r\leq K)$ as follows.

   Divide $(t_{r-1}',t_r')_i$ into $L+1$ intervals of equal length $\frac{1}{K(L+1)}$ by points
   $$
   t_{r-1}'=\gamma_0<\gamma_1<\cdots<\gamma_{L+1}=t_r'.
   $$
   Comparing $\#\big(Sp\phi_{x_0}\cap (\gamma_k,\gamma_{k+1})_i\big)$ for $k=0,1,\cdots,L$, choose $0\leq k_0\leq L$ such that
   $$
   \#\big(Sp\phi_{x_0}\cap(\gamma_{k_0},\gamma_{k_0+1})_i\big)\leq\#\big(Sp\phi_{x_0}\cap(\gamma_k,\gamma_{k+1})_i\big)
   $$
   for each $0\leq k\leq R-1$. Set $S_r=\gamma_{k_0}$ and $\widetilde{S_r}=\gamma_{k_0+1}$. Then
   $$
   \#\big(Sp\phi_{x_0}\cap(S_r,\widetilde{S_r})_i\big)\leq\frac{1}{L+1}\#\big(Sp\phi_{x_0}\cap(t_{r-1},t_r)_i\big).
   $$
   Hence
   $$
   \sum_{r=1}^{K}\#\big(Sp\phi_{x_0}\cap(S_r,\widetilde{S_r})_i\big)\leq \frac{1}{L+1}\#(Sp\phi_{x_0}\cap(0,1)_i)=\frac{1}{L+1}\#(Sp\phi_{x_0}\cap (0,1)_i).
   $$
   Set $\widetilde{V}=\cup_{j=1}^K(S_j,\widetilde{S_j})$, then
   $$
   \#\big (Sp\phi_{x_0}\cap [0,1]_i \backslash \widetilde{V}\big)\geq \frac{L}{L+1}\#(Sp\phi_{x_0}\cap (0,1)_i).
   $$
   Since $[0,1]_i \backslash \widetilde{V}$ is a disjoint union of closed intervals $V_0^i,V_1^i,V_2^i,\cdots,V_K^i$, we note that some of them may be single point. Actually, we denote them of the form
   $$
   V_0^i=[0,S_1]_i,V_1^i=[\widetilde{S_1},S_2]_i,\cdots,V_K^i=[\widetilde{S_K},1]_i.
   $$

   Condition (iii) has already be verified. Let us verify Conditions (i) and (ii).

   Since $\widetilde{S_r}-S_r=\dfrac{1}{K(L+1)}$,
   $$
   dist(V_{r_1}^i,V_{r_2}^i)\geq\frac{1}{K(L+1)},\quad if\quad r_1\neq r_2.
   $$
   From the definition of $W_r^i$, we know that
   $$
   dist(W_{r_1}^i,W_{r_2}^i)\geq distance(V_{r_1}^i,V_{r_2}^i)-\frac{2}{4K(L+1)}\geq\frac{1}{2K(L+1)}.
   $$
   Hence $W_{r_1}^i\cap W_{r_2}^i=\varnothing$ if $r_1\neq r_2$. This means (i) holds.

   For (ii), notice that for each $r$
   $$
   t_{r-1}'+\frac{1}{K(L+1)}\leq \widetilde{S_r}\leq S_{r+1}\leq  t_{r+1}'-\frac{1}{K(L+1)}.
   $$
   Hence,
   $$
   diameter (V_r^i)\leq\frac{2}{K}- \frac{2}{K(L+1)}
   $$
   It follows that
   $$
   diameter(W_r^i)<\frac{2}{K}.
   $$
   So (ii) holds.

   For each $i\in \{1,2,\cdots,l\}$, we have a collection of disjoint closed intervals
   $$
   V_0^1,V_1^1,\cdots,V_{K}^1,\cdots,V_0^l,V_1^l,\cdots,V_K^l
   $$
   that satisfies conditions (i), (ii), (iii), and $V_r^i\subset [\eta,1-\eta]_i$ for all $r=\{1,2,\cdots,K-1\}$.

   For each $h\in H(\eta)$, there are mutually orthogonal projections $p_1(h)$, $p_2(h)$, $\cdots$, $p_{m(h)}(h)\in B$
   and real numbers $\lambda_1(h),\lambda_2(h),\cdots,\lambda_{m(h)}(h)$ such that
   $$
   \|\phi(h)-\sum_{k=1}^{m(h)}\lambda_k(h)p_k(h)\|<\frac{1}{6}.\eqno (*)
   $$
   Denote
   $$
   \Lambda(h)=\{\lambda_1(h)^{\thicksim{\rm rank}\,p_1(h)},\cdots,\lambda_{m(h)}(h)^{\thicksim{\rm rank}\,\,p_{m(h)}(h)},
   0^{\sim l'-\sum_{k=1}^{m(h)}{\rm rank}\,p_k(h)}\},
   $$
   $$
   \Lambda_1(h)=\{\lambda | \lambda\in \Lambda(h),\,\,\lambda\in(1-\frac{1}{6},\,\,1]\},
   $$
   $$
   \Lambda_2(h)=\{\lambda | \lambda\in \Lambda(h),\,\,\lambda\in(1-\frac{1}{2},1-\frac{1}{6}]\},
   $$
   $$
   \Lambda_3(h)=\{\lambda | \lambda\in \Lambda(h) ,\,\,\lambda\in(1-\frac{5}{6},1-\frac{1}{2}]\}.
   $$
   Then for any $x\in [0,1]\in Sp(B)$, $Eg(\phi_{x}(h))$ and $\Lambda(h)$ can be paired to within $\frac{1}{6}$ one by one, then
    $Eg(\phi_{x_0}(h))$ and $Eg(\phi_{x}(h))$ can be paired to within $\frac{1}{3}$ one by one for each $h\in H(\eta)$. We will
    use these pairing results frequently.

   $ Step\,2 $:  %%%%%%%%%%%%%%%%%%%%%%%%%%%%%%%%%%%%%%%第二步对边界进行讨论
   Set $c_r^i=\#(Sp\phi_{x_0}\cap V_r^i)$, for all $i\in \{1,2,\cdots,l\},\,r\in \{0,1,2,\cdots,K\}$,
   let us construct $P_r^i$ for  each $i\in \{1,2,\cdots,l\},\,r=\{1,2,\cdots,K-1\}$.

  For each $x\in [0,1]\subset Sp(B)$, we have
  $$
  \phi_x(f,a)=U_x^*\cdot\big( a(\theta_1)^{\sim t_1(x)},\cdots,a(\theta_p)^{\sim t_p(x)},f(y_1(x)),f(y_2(x)),\cdots,f(y_{\bullet}(x))\big)\cdot U_x.  $$
  Rewrite it by
  $$
  \phi_x(f,a)=U_x^*\cdot
  \left(
  \begin{array}{cccccc}
    a(\theta_1)\otimes I_{t_1(x)} &  &  &   &  &    \\
    \\
    \quad\quad\quad\ddots &  &   &   &   &   \\
    \\
    & a(\theta_p)\otimes I_{t_p(x)}  &   &   &    \\
    \\
     &  & f(y_1(x))&  &   \\
   \\
    &  &  &  & \ddots  &  \\
    \\
    &  &   &  &   & f(y_{\bullet}(x))

  \end{array}
  \right)
  \cdot U_x,
  $$
  where $y_1(x),\cdots,y_{\bullet}(x)\in \coprod_{i=1}^{l}(0,1)_i\subset Sp(A)$ and $U_x,\,t_j(x),\,y_{\bullet}(x)$ may not depend on $x$ continuously.

  For any subset $T\subset (0,1)_i\subset Sp(A)$ and $x\in [0,1]\subset Sp(B)$,
  define $E_T(x)$ as follows: let $a(\theta_j)=0$ for each $j\in \{1,2,\cdots,p\}$, and for $y\in \coprod_{i=1}^{l} (0,1)_i$, let
  $$
  f_T(y)=
  \begin{cases}
    0_{l_s}, & \mbox{if } y\in (0,1)_s,\ s\neq i,  \\
    0_{l_i}, & \mbox{if } y\in (0,1)_i,\ y\notin T,  \\
    I_{l_i}, & \mbox{if } y\in (0,1)_i,\ y\in T.
  \end{cases}
  $$
  Set
  $$
  E_T(x)=U_x^*\cdot
  \left(
  \begin{array}{ccccccc}
    0&  &  & &  &  &     \\
     & 0 &  & &  &  &   \\
     &  & \ddots & &  &   & \\
     &  &  & 0 &  & &   \\
     &  & &  & f_T(y_1(x)) &  &    \\
     \\
      & &  &  &  & \ddots &  \\
     \\
      & &  &   &  &   & f_T(y_{\bullet}(x))

  \end{array}
  \right)
  \cdot
  U_x.
  $$
  $E_T(x)$ does not depend on $x$ continuously, but $E_T(x)$ commutes with $\phi_x(f)$ for all $f\in A$, if $T_1\cap T_2=\varnothing$, then $E_{T_1}(x)E_{T_2}(x)=0$.

  Recall that $\{e_{ss'}^i\}\,(1\leq i\leq l,1\leq s,s' \leq l_i)$ are matrix units of $F_2$. Let $T_i=[\eta,1-\eta]_i\subset Sp(A)$, choose a positive function $\rho$ in $C[0,1]$ with $\rho|_{[\eta,1-\eta]}=1$ and $\rho(0)=\rho(1)=0$, obviously,
  $\rho e_{ss'}^i\in A$ and the elements
  $$
  \overline{e}_{ss'}^i(x):=E_{T_i}(x)\cdot\phi_x(\rho\cdot e_{ss'}^i)\cdot E_{T_i}(x)
  $$
  fulfill the canonical relations for matrix units:
  $$
  \overline{e}_{s_1s_1'}^i(x)\cdot\overline{e}_{s_2s_2'}^i(x)=\delta_{s_1's_2}\cdot\overline{e}_{s_1s_2'}^i(x)
  $$
  for all $1\leq s_1,s_2,s_1',s_2'\leq l_i$.

   Let $h_r^i$ be the test function corresponding to $V_r^i$. From the pairing results,
   for any $x_1,x_2,x_3,x_4\in [0,1]\subset Sp(B)$, then we have

   \begin{eqnarray*}
     \#\big(Sp\phi_{x_1}\cap V_r^i\big) &\leq& \#(\Lambda_1({h_r^i})) \\
      &\leq & \#\big(Sp\phi_{x_2}\cap B(V_r^i,\frac{1}{3}\eta)\big) \\
      &\leq& \#(\Lambda_1({h_r^i})\cup\Lambda_2({h_r^i})) \\
      &\leq& \#\big(Sp\phi_{x_3}\cap B(V_r^i,\frac{2}{3}\eta)\big) \\
      &\leq& \#\big(\Lambda_1({h_r^i})\cup\Lambda_2({h_r^i})\cup\Lambda_3({h_r^i})\big) \\
      &\leq& \#\big(Sp\phi_{x_4}\cap B(V_r^i,\eta)\big)
   \end{eqnarray*}

   Consider the following two cases:

   Case 1.$\quad$ $\Lambda_2({h_r^i})=\Lambda_3({h_r^i})=\varnothing$

   In this case, we have $\#(\Lambda_1({h_r^i}))=\#(\Lambda_1({h_r^i})\cup\Lambda_2({h_r^i})\cup\Lambda_3({h_r^i}))$, then
   $$
   \#\big(Sp\phi_{x}\cap B(V_r^i,\frac{2}{3}\eta)\big)=\#\big(Sp\phi_x\cap B(V_r^i,\frac{1}{3}\eta)\big).
   $$
   Therefore,
   $$
   Sp\phi_x\cap B(V_r^i,\frac{2}{3}\eta)\backslash B(V_r^i,\frac{1}{3}\eta)=\varnothing.
   $$

   Define $\chi_r^i$ as follows: let $a(\theta_j)=0$ for each $j\in \{1,2,\cdots,p\}$, and let
  $$
  \chi_r^i(t)=(\chi_r^i(t,1),\chi_r^i(t,2),\cdots,\chi_r^i(t,l)),
  $$
  where $\chi_r^i(t,s)=0_{l_s}$, for $s\neq i$, and
  $$
  \chi_r^i(t,i)=
  \begin{cases}
    0_{l_i}, & \mbox{if }\,\, t\in [0,1]_i\backslash\overline{B(V_r^i,\frac{2}{3}\eta)},  \\
     linear, & \mbox{if }\,\, t\in \overline{B(V_r^i,\frac{2}{3}\eta)}\backslash B(V_r^i,\frac{1}{3}\eta),\\
    I_{l_i}, & \mbox{if }\,\, t\in  \overline{B(V_r^i,\frac{1}{3}\eta)}.
  \end{cases}
  $$
   %by $\chi_r^i|_{\overline{B(V_r^i,\frac{1}{3}\eta)}}=1$,\, $support(\chi_r^i)\subset\overline{B(V_r^i,\frac{2}{3}\eta)}$.
   Then $\chi_r^i$ is an element in the center of $A$ and $\phi(\chi_r^i)$ is a projection in $B$.
   Let $P_r^i=\phi(\chi_r^i)$, then ${\rm rank}\,P_r^i=l_i\#(\Lambda_1({h_r^i}))\geq l_ic_r^i$  and clearly we have
   $$
   P_r^i\phi(f)=\phi(\chi_r^i)\phi(f)=\phi(f)\phi(\chi_r^i)=\phi(f)P_r^i,
   $$
   for all $f\in A$.

   Case 2.$\quad$ $\Lambda_2({h_r^i})\cup\Lambda_3({h_r^i})\neq\varnothing$

   Denote $U_r^i=\overline{B(V_r^i,\eta)}$, then $U_r^i\subset W_r^i$, we have
   $$
   \#\big(Sp\phi_{x}\cap U_r^i\big)\geq\#\big(\Lambda_1({h_r^i})\big)+1\geq c_r^i+1.
   $$
   for any $x\in [0,1]\subset Sp(B)$.

   We will apply lemma \ref{2.7} to $U_r^i$ and any $x'\in [0,1]$, there exists $\eta(x'),\,\delta(x')>0$ small enough(assume that for all $x'\in [0,1],\,\eta(x')\leq \frac{1}{3}\eta$ and $\eta(0)=\eta(1))$ such that
   $$
   \#\big (Sp\phi_x\cap B(U_r^i,2\eta(x'))\big)=\#\big (Sp\phi_{x'}\cap U_r^i\big)
   $$
   and
   $$
   Sp\phi_x\cap \overline{B(U_r^i,3\eta(x'))}\backslash B(U_r^i,2\eta(x'))=\varnothing
   $$
   for all $x\in B(x',\delta(x'))$.

  Define $\chi_r^i$ corresponding to the point $x'\in Sp(B)$ as follows: let $a(\theta_j)=0$ for each $j\in \{1,2,\cdots,p\}$, and let
  $$
  \chi_r^i(t)=(\chi_r^i(t,1),\chi_r^i(t,2),\cdots,\chi_r^i(t,l)),
  $$
  where $\chi_r^i(t,s)=0_{l_s}$, if $s\neq i$, and
  $$
  \chi_r^i(t,i)=
  \begin{cases}
  0_{l_i}, & \mbox{if }\,\, t\in [0,1]_i\backslash\overline{B(U_r^i,3\eta(x'))},  \\
  linear, & \mbox{if }\,\,t\in \overline{B(U_r^i,3\eta(x'))}\backslash B(U_r^i,2\eta(x')),\\
   e_{11}^i & \mbox{if }\,\, t\in \overline{B(U_r^i,2\eta(x'))}.
  \end{cases}
  $$
  here we use $e_{11}^i$ as the matrix units of $M_{l_i}(\mathbb{C})$.

  Then ${\chi_r^i}\in A$ and $\phi({\chi_r^i})$ defines a continuous projection-valued function on a certain open neighbourhood $B(x',\delta(x'))$ and denote it by ${q_r^i}|_{B(x',\delta(x'))}(x)$, then
  $$
  {\rm rank}\,{q_r^i}|_{B(x',\delta(x'))}(x)\geq \#\big (Sp\phi_x\cap B(U_r^i,\eta(x'))\big)\geq \#\big(Sp\phi_{x}\cap U_r^i\big)\geq \#\big(\Lambda_1({h_r^i})\big)+1
  $$
  for any $x\in B(x',\delta(x'))$.

   Let $\mathcal{W}=\{B(x',\delta(x'),x'\in [0,1]\}$, then $\mathcal{W}$ is an
   open cover of $[0,1]$, apply lemma \ref{3.1} (selection principle), there exists a projection-valued function $p_r^i$ defined on the whole set $[0,1]$ with ${\rm rank}\ p_r^i=\#(\Lambda_1(h_r^i))$, and
   $$
   p_r^i(x)\leq\bigvee\{{q_r^i}|_{B(x',\delta(x'))}(x)\,\,|B(x',\delta(x')\in\mathcal{W}\}\leq\overline{e}_{11}^i(x).
   $$

   In general, $p_r^i(x)$ belongs to $C([0,1],M_{l'}(\mathbb{C}))$, not to $B$, and therefore, we need to make suitable changes
   near the endpoints. We require that 0 and 1 each belong to only  one of the open sets in the cover in $\mathcal{W}$,
   say $O_0(=B(0,\delta(0)))$ and $O_1(=B(0,\delta(1)))$, respectively. Since
   we assume that $\eta(0)=\eta(1)$, then the elements $\chi_r^i$ corresponding to 0 and $\chi_r^i$ corresponding to 1 are the same,
   denote this element by $\widetilde{\chi}_r^i$  and  we have
   $$
   {q_r^i}|_{O_0}(0)=\pi_0\circ\phi({\widetilde{\chi}_r^i})\,\,\,{\rm and}\,\,\,\pi_1\circ\phi({\widetilde{\chi}_r^i})={q_r^i}|_{O_1}(1).
   $$

   Denote
   $$
   \overline{p}_r^i=\sum_{\lambda_k(h_r^i)\in \Lambda_1(h_r^i)}\,p_k(h_r^i).
   $$
   Then $\overline{p}_r^i\in B$ (see notation in ($*$) and the paragraph below), and
   $$
   {\rm  rank}\,\,\overline{p}_r^i = {\rm rank}\,\, p_r^i=\#(\Lambda_1(h_r^i)).
   $$
   Recall that for each $j'\in\{1,2,\cdots,p'\}$ and $a'=(a'(\theta_1'),a'(\theta_2'),\cdots,a'(\theta_{p'}'))\in F_1'$,
   we have defined $\widetilde{\pi}_e^{j'}:B\rightarrow  M_{k_{j'}'}(\mathbb{C})$, $\widetilde{\pi}_0^{j'}: \varphi_0'(F_1')\rightarrow M_{k_{j'}'}(\mathbb{C})$,\, $\widetilde{\pi}_1^{j'}: \varphi_1'(F_1')\rightarrow M_{k_{j'}'}(\mathbb{C})$, as follows:
   $$
   \widetilde{\pi}_e^{j'}:B\rightarrow  M_{k_{j'}'}(\mathbb{C}),
   $$
   $$
   \widetilde{\pi}_0^{j'}\circ\varphi_0'(a')={\rm sgn}(\alpha_{j'}')\cdot a'(\theta_{j'}'),
   $$
   $$
   \widetilde{\pi}_1^{j'}\circ\varphi_1'(a')={\rm sgn}(\beta_{j'}')\cdot a'(\theta_{j'}'),
   $$
   where sgn$(x)$ is the sign function.

   Since $\overline{p}_r^i(0),\,{q_r^i}|_{O_0}(0)\in \varphi_0'(F_1'),\,\,
   \overline{p}_r^i(1),\,{q_r^i}|_{O_1}(1)\in \varphi_1'(F_1')$, then

   $$
   {\rm rank}\,\,\widetilde{\pi}_0^{j'}(\overline{p}_r^i(0))\leq \#(Sp(\widetilde{\pi}_0^{j'}\circ\pi_0\circ\phi)\cap U_r^i)
   \leq{\rm rank}\,\,\widetilde{\pi}_0^{j'}({q_r^i}|_{O_0}(0))
   $$
   $$
   {\rm rank}\,\,\widetilde{\pi}_1^{j'}(\overline{p}_r^i(1))\leq \#(Sp(\widetilde{\pi}_1^{j'}\circ\pi_1\circ\phi)\cap U_r^i)
   \leq{\rm rank}\,\,\widetilde{\pi}_1^{j'}({q_r^i}|_{O_1}(1))
   $$
   holds for each $j'\in \{1,2,\cdots,p'\}$ and if  $j'$ satisfies $\alpha_{j'}'>0,\,\beta_{j'}'>0$, then
   $$
   \widetilde{\pi}_0^{j'}(\overline{p}_r^i(0))=\widetilde{\pi}_1^{j'}(\overline{p}_r^i(1))=\pi_e^{j'}(\overline{p}_r^i)
   \in M_{k_{j'}'}(\mathbb{C})
   $$
   and
   $$
   \widetilde{\pi}_0^{j'}({q_r^i}|_{O_0}(0))=\widetilde{\pi}_1^{j'}({q_r^i}|_{O_1}(1))=\pi_e^{j'}\circ\phi(\widetilde{\chi}_r^i)\in M_{k_{j'}'}(\mathbb{C}).
   $$

   Then there exists a collection of unitaries $u_{j'}\in M_{k_{j'}'}(\mathbb{C})$, such that
   $$
   u_{j'}^*\cdot\widetilde{\pi}_0^{j'}(\overline{p}_r^i(0))\cdot u_{j'}<\widetilde{\pi}_0^{j'}({q_r^i}|_{O_0}(0)),
   $$
   $$
   u_{j'}^*\cdot\widetilde{\pi}_0^{j'}(\overline{p}_r^i(0))\cdot u_{j'}<\widetilde{\pi}_1^{j'}({q_r^i}|_{O_1}(1)).
   $$
   Hence,
   $$
   \varphi_0'(u_1^*,u_2^*,\cdots,u_{p'}^*)\cdot\overline{p}_r^i(0)\cdot\varphi_0'(u_1,u_2,\cdots,u_{p'})<{q_r^i}|_{O_0}(0),
   $$
   $$
   \varphi_1'(u_1^*,u_2^*,\cdots,u_{p'}^*)\cdot\overline{p}_r^i(1)\cdot\varphi_1'(u_1,u_2,\cdots,u_{p'})<{q_r^i}|_{O_1}(1).
   $$
   Connect $\varphi_0'(u_1,u_2,\cdots,u_{p'})$ and $\varphi_1'(u_1,u_2,\cdots,u_{p'})$ by a unitary path $v(t)\in B$, then
   $v^*\cdot\overline{p}_r^i\cdot v$ belongs to $B$.

   Since
   $$
   p_r^i(x)<{q_r^i}|_{O_0}(x),\,\,\forall x\in O_0,
   $$
   $$
   p_r^i(x)<{q_r^i}|_{O_1}(x),\,\,\forall x\in O_1,
   $$
   fix $\epsilon'>0$ small enough, then we can connect $v^*(0)\cdot\overline{p}_r^i(0)\cdot v(0)$ and $p_r^i(\epsilon')$ by a path
   new $p_r^i(x),\,0\leq x\leq\epsilon'\leq\delta(0)$ with the property that new $p_r^i(x)<{q_r^i}|_{O_0}(x)$. A similar construction can be carried out in $1-\delta(1)\leq 1-\epsilon'\leq x\leq 1$. This will ensure that new $p_r^i$ belongs to $B$.

   Set
   $$
   P_r^i(x)=\sum_{s=1}^{l_i}\overline{e}_{s1}^i(x)\cdot{\rm new}\,p_r^i(x)\cdot\overline{e}_{1s}^i(x);
   $$
   Now we have $P_r^i(x)\leq E_{W_r^i}(x)$.

   Choose arbitrary $z_r^i\in W_r^i$, if we change all the spectra in $Sp(\phi_x)\cap W_r^i$ to $z_r^i$, we obtain a pointwise homomorphism $\widetilde{\phi}_x$.

   Since Diameter $(W_r^i)<\frac{2}{K}$ and for any $x,x'\in [0,1]_i \subset Sp(A)$ with $d(x,x')<\frac{2}{K}$,
   $\|g(x)-g(x')\| <\frac{1}{J}$ holds for all $g\in G$.
   Then we have that
   $$
   \|\phi_x(g)-\widetilde{\phi}_{x}(g)\|<\frac{1}{J}
   $$
   holds for all $x\in [0,1]\subset Sp(B)$, $g\in G$.

   Let $g\in G$ be written as
   $$
   g=\sum_{i=1}^{l}\sum_{s,s'=1}^{l_i}g_{ss'}^ie_{ss'}^i,
   $$
   where $g_{ss'}^i\in C[0,1]$.

   From the definition of $E_{W_r^i}(x)$, we have
   \begin{eqnarray*}
     \widetilde{\phi}_x(g)P_r^i(x)&=& \sum_{iss'}\phi_x(\rho\cdot g_{ss'}^ie_{ss'}^i)(x)\sum_{s=1}^{l_i}\overline{e}_{s1}^i\cdot{\rm new}\,p_r^i(x)\cdot\overline{e}_{1s}^i(x) \\
     &=& \sum_{ss'}\overline{e}_{s1}^i(x)\cdot\phi_x(\rho\cdot g_{ss'}^i)\cdot{\rm new}\,p_r^i(x)\cdot\overline{e}_{1s'}^i(x) \\
     &=& \sum_{ss'}\overline{e}_{s1}^i(x)\cdot{\rm new}\,p_r^i(x)\cdot\widetilde{\phi}_x(\rho\cdot g_{ss'}^i)\cdot\overline{e}_{1s'}^i(x) \\
     &=& \sum_{s=1}^{l_i}\overline{e}_{s1}^i(x)\cdot{\rm new}\,p_r^i(x)\cdot\overline{e}_{1s}^i(x)\sum_{iss'}\widetilde{\phi}_x(\rho\cdot g_{ss'}^ie_{ss'}^i)(x)\\
     &=&  P_r^i(x)\widetilde{\phi}_x(g).
   \end{eqnarray*}
   Therefore,
   $$
   \|P_r^i\phi(g)-\phi(g)P_r^i\|<\frac{2}{J}
   $$
   holds for all $g\in G$.

   From the two cases above, we have constructed a collection of mutually orthogonal projections
   $$
   P_1^1,\cdots,P_{K-1}^1,\cdots,P_1^l,\cdots,P_{K-1}^l,
   $$
   with ${\rm rank}\,P_r^i=\#\Lambda_1(h_j^i)$ for all $i\in \{1,2,\cdots,l\},\,r\in \{1,2,\cdots,K-1\}$.

   $Step\,3$:  %处理边界的问题，比处理中间的部分要复杂，复杂在测试函数的选取
   Now we deal with $V_0^1,V_K^1,\cdots,V_0^l,V_K^l$. Since
   $$
   0_i\sim {\rm diag}\{\theta_1^{\thicksim\alpha_{i1}},\theta_2^{\thicksim\alpha_{i2}},\cdots,\theta_p^{\thicksim\alpha_{ip}}\},
   $$
   $$
   1_i\sim {\rm diag}\{\theta_1^{\thicksim\beta_{i1}},\theta_2^{\thicksim\beta_{i2}},\cdots,\beta_p^{\thicksim\beta_{ip}}\}.
   $$
   $\theta_j$ may appear in different end points, this fact force us to define subsets as follows:

   Denote
   $$
   \widetilde{V_j}=\bigcup_{\alpha_{ij}\neq0}V_0^i\cup\bigcup_{\beta_{ij}\neq0}V_K^i
   $$
   $$
   \widetilde{W_j}=\bigcup_{\alpha_{ij}\neq0}W_0^i\cup\bigcup_{\beta_{ij}\neq0}W_K^i
   $$
   for each $j\in \{1,2,\cdots,p\}\,,i\in\{1,2,\cdots,l\}$. Then we turn these $2l$ intervals to $p$ subsets.

   Set $X_j=\{\theta_j\}$, we will construct $\widetilde{P_j}$ corresponding to $X_j=\{\theta_j\}$ and $\widetilde{V_j}$, for $j=1,2,\cdots,p$.

   Now we will replace the spectra in $\bigcup_{i=1}^lW_0^i$ by
   $$
   0_i\sim {\rm diag}\{\theta_1^{\thicksim\alpha_{i1}},\theta_2^{\thicksim\alpha_{i2}},\cdots,\theta_p^{\thicksim\alpha_{ip}}\}
   $$
   and replace the spectra in $\bigcup_{i=1}^lW_K^i$ by
   $$
   1_i\sim {\rm diag}\{\theta_1^{\thicksim\beta_{i1}},\theta_2^{\thicksim\beta_{i2}},\cdots,\beta_p^{\thicksim\beta_{ip}}\}
   $$
   to obtain a new homomorphism $\phi'$ from $\phi$ at each point.

   Obviously,
   $$
   \|\phi_x'(g)-\phi_x(g)\|<\frac{1}{J}
   $$
   holds for all $g\in G$,\,$x\in [0,1]$.

   Then for any given $x\in [0,1]$, there are unitaries
   $U_x$, $V_x$ such that
   $$
   \phi_x'(f,a)=U_x^*V_x^*\cdot
  \left(
  \begin{array}{cccccc}
    a(\theta_1)\otimes I_{t_1'(x)} &  &  &   &  &    \\
    \\
    \quad\quad\quad\ddots &  &   &   &   &   \\
    \\
    & a(\theta_p)\otimes I_{t_p'(x)}  &   &   &    \\
    \\
     &  & f(y_1(x))&  &   \\
   \\
    &  &  &  & \ddots  &  \\
    \\
    &  &   &  &   & f(y_{\bullet}(x))

  \end{array}
  \right)
  \cdot V_xU_x,
  $$
   where $y_1(x),\cdots,y_{\bullet}(x)\in \coprod_{i=1}^{l}[0,1]_i\backslash(W_0^i\cup W_K^i)\subset Sp(A)$ and $U_x,\,t_j(x),\,y_{\bullet}(x)$ may not depend on $x$ continuously.

   Define $E_{\widetilde{W_j}}(x)$ follows: let $f_{\widetilde{W_j}}(y)=0_{l_i}$, if $y\in (0,1)_i\backslash(W_0^i\cup W_K^i)$  for each $i\in \{1,2,\cdots,l\}$, and
   $$
   a(\theta_s)=
   \begin{cases}
    0_{k_s}, & \mbox{if } s\neq j,  \\
    I_{k_j}, & \mbox{if } s=j.
  \end{cases}
  $$
  Set
  $$
  E_{\widetilde{W_j}}(x)=U_x^*V_x^*\cdot
  \left(
  \begin{array}{ccccccc}
   0  &  &  &   &  &  &    \\
    & \ddots &   &   &   &  &   \\
  &   & I_{k_j}\otimes I_{t_j'(x)}  &   &   &   &   \\
   &  &  & \ddots &  &   & \\
  &  &  & & 0&  &    \\
  &  &  &  &  & \ddots  &   \\
  &  & &    &  &   & 0
  \end{array}
  \right)
  \cdot V_xU_x,
  $$
  it does not depend on $x$ continuously, but $E_{\widetilde{W_j}}(x)$ commutes with $\phi_x'(f)$ for all $f\in A$, and if $j_1\neq j_2$, then $E_{\widetilde{W_{j_1}}}(x)E_{\widetilde{W_{j_2}}}(x)=0$.

   Suppose that $\{f_{ss'}^j\}\,(1\leq j\leq p,1\leq s,s'\leq k_j)$ are matrix units of $F_1$, for each $t\in[0,1]_i,\,i\in \{1,2,\cdots,l\}$.
   Let
   $$
   g_{ss'}^j(t)=(g_{ss'}^j(t,1),g_{ss'}^j(t,2),\cdots,g_{ss'}^{j}(t,l))
   $$
   where
   $$
   g_{ss'}^j(t,i)=
   \begin{cases}
   \varphi_0^i(f_{ss'}^j)\dfrac{\eta-dist(t,[0,\frac{1}{K}]_i)}{\eta}, & \mbox{if } 0\leq t\leq \frac{1}{K}+\eta \\
   0, & \mbox{if } \frac{1}{K}+\eta\leq t\leq  1-\frac{1}{K}-\eta\\
   \varphi_1^i(f_{ss'}^j)\dfrac{\eta-dist(t,[1-\frac{1}{K},1]_i)}{\eta}, & \mbox{if } 1-\frac{1}{K}-\eta\leq t\leq 1
   \end{cases}
   $$
   Obviously,  we have
   $$
   \phi_x'(g_{ss'}^j)=\phi_x(g_{ss'}^j).
   $$
   holds for each $x\in [0,1]\subset Sp(B)$ and $1\leq j\leq p,1\leq s,s'\leq k_j$.

   Define
   $$
   \overline{f}_{ss'}^j(x):=E_{\widetilde{W_j}}(x)\cdot\phi_x'(g_{ss'}^j)\cdot E_{\widetilde{W_j}}(x).
   $$
   We also have
   $$
   \overline{f}_{s_1s_1'}^j(x)\cdot\overline{f}_{s_2s_2'}^j(x)=\delta_{s_1's_2}\cdot\overline{f}_{s_1s_2'}^j(x)
   $$
   for all $1\leq s_1,s_1',s_2,s_2'\leq k_j$.

   Let $h_j$ be the test function corresponding to $\widetilde{V_j}$ and $X_j$. In a similar way, from the pairing results, for any $x_1,x_2,x_3,x_4\in [0,1]$, we have
   \begin{eqnarray*}
   \#_{X_j}\big(Sp\phi_{x_1}\cap \widetilde{V_j}\big) &\leq& \#(\Lambda_1({h_j})) \\
   &\leq &\#_{X_j}\big(Sp\phi_{x_2}{\cap} B(\widetilde{V_j},\frac{1}{3}\eta)\big) \\
   &\leq& \#(\Lambda_1({h_j})\cup\Lambda_2({h_j})) \\
   &\leq& \#_{X_j}\big(Sp\phi_{x_3}{\cap} B(\widetilde{V_j},\frac{2}{3}\eta)\big) \\
   &\leq& \#\big(\Lambda_1({h_j})\cup\Lambda_2({h_j})\cup\Lambda_3({h_j})\big) \\
   &\leq& \#_{X_j}\big(Sp\phi_{x_4}\cap B(\widetilde{V_j},\eta)\big).
   \end{eqnarray*}

   We still consider the following two cases:

   Case 1.$\quad$ $\Lambda_2({h_j})=\Lambda_3({h_j})=\varnothing$

   This case is just similar to case 1 in step 2, we have
   $$
   \#_{X_j}\big(Sp\phi_x\cap B(\widetilde{V_j}, \frac{1}{3}\eta)\big)=\#_{X_j}\big(Sp\phi_x\cap B(\widetilde{V_j}, \frac{2}{3}\eta)\big).
   $$
   Therefore,
   $$
   Sp\phi_x\cap B(\widetilde{V_j},\frac{2}{3}\eta)\backslash Sp\phi_x\cap B(\widetilde{V_j},\frac{1}{3}\eta)=\varnothing.
   $$
   Recall that
   $$
   \widetilde{V_j}=\bigcup_{\alpha_{ij}\neq0}V_0^i\cup\bigcup_{\beta_{ij}\neq0}V_K^i
   $$
   Then, if $i$ satisfies $\alpha_{ij}\neq 0$, we will have
   $ Sp\phi_x{\cap}(B(V_0^i,\frac{2}{3}\eta)\backslash B(V_0^i,\frac{1}{3}\eta))=\varnothing,$
   and if $i$ satisfies $\beta_{ij}\neq 0$, we will have $Sp\phi_x{\cap}\big(B(V_K^i,\frac{2}{3}\eta)\backslash B(V_K^i,\frac{1}{3}\eta)\big)=\varnothing.$

   Define $\chi_j$ as follows: let $a=(a(\theta_1),a(\theta_2),\cdots,a(\theta_p))\in F_1$, where $a(\theta_j)=I_{k_j}$ and $a(\theta_s)=0_{k_s}$, if $s\neq j$. Set
   $$
  \chi_j(t)=(\chi_j(t,1),\chi_j(t,2),\cdots,\chi_j(t,l)),
   $$
   where
   $$
   \chi_j(t,i)=
   \begin{cases}
    {\rm sgn}(\alpha_{ij})\cdot\varphi_0^i(a),\,\, & \mbox{if }\,\,t\in \overline{ B(V_0^i,\frac{1}{3}\eta)},  \\
    linear, & \mbox{if }\,\, t\in \overline{B(V_0^i,\frac{2}{3}\eta)}\backslash B(V_0^i,\frac{1}{3}\eta),\\
    0_{l_i}, & \mbox{if }\,\, t\in [0,1]_i\backslash\big(B(V_0^i,\frac{2}{3}\eta){\cup} B(V_K^i,\frac{2}{3}\eta)\big), \\
    linear, & \mbox{if }\,\, t\in \overline{ B(V_K^i,\frac{2}{3}\eta)}\backslash B(V_K^i,\frac{1}{3}\eta),\\
    {\rm sgn}(\beta_{ij})\cdot\varphi_1^i(a), & \mbox{if }\,\, t\in \overline{B(V_K^i,\frac{1}{3}\eta)}.
    \end{cases}
   $$
   Then $\chi_j\in A$ and $\phi(\chi_j)$ is a projection in $B$. Set $P_j=\phi(\chi_j)$,  then
   $P_j$ commutes with $\phi(f)$ for all $f\in A$ and
   $$
   {\rm rank}\,P_j=k_j\#(\Lambda_1({h_j}))\geq k_jt_j+k_j\sum_{i=1}^{l}\alpha_{ij}c_0^i+k_j\sum_{i=1}^{l}\beta_{ij}c_k^i.
   $$

   Case 2.$\quad$ $\Lambda_2({h_j})\cup\Lambda_3({h_j})\neq\varnothing$

   Denote $U_j\triangleq \overline{B(\widetilde{V_j},\eta)}$, then we have
   $$
   \#\big(Sp\phi_{x}\cap U_j\big)\geq\#\big(\Lambda_1({h_j})\big)+1\geq t_i+\sum_{i=1}^{l}\alpha_{ij}c_0^i+\sum_{i=1}^{l}\beta_{ij}c_k^i+1.
   $$
   for any $x\in [0,1]$.

   Apply lemma \ref{2.7} to $U_j$ and any $x'\in [0,1]$, we may require that $\eta'(x'),\delta'(x')$ small enough (we still assume
   that for all $x'\in [0,1],\,\eta'(x')\leq \frac{1}{3}\eta$ and $\eta'(0)=\eta'(1))$ and also satisfy
   $$
   \#_{X_j}\big(Sp\phi_{x}\cap \overline{B(U_j,2\eta'(x'))}\big)=\#_{X_j}\big(Sp\phi_{x_0}\cap U_j\big),\,\,
   Sp\phi_x\cap \overline{B(U_j,3\eta'(x'))}\backslash B(U_j,2\eta'(x'))=\varnothing
   $$
   for all $x\in B(x',\delta'(x'))$.

   Define $\chi_j$ corresponding to the point $x'$ as follows: Set
   $$
   \chi_j(t)=(\chi_j(t,1),\chi_j(t,2),\cdots,\chi_j(t,l)),
   $$
   where
   $$
   \chi_j(t,i)=
   \begin{cases}
    {\rm sgn}(\alpha_{ij})\cdot\varphi_0^i(f_{11}^j),\,\, & \mbox{if }\,\,t\in [0,\frac{1}{K}]_i{\cap}\overline{B(U_j,2\eta'(x'))},  \\
    linear, & \mbox{if }\,\, t\in [0,\frac{1}{K}]_i{\cap}\overline{B(U_j,3\eta'(x'))}\backslash B(U_j,2\eta'(x')),\\
    0_{l_i}, & \mbox{if }\,\, t\in [0,1]_i\backslash B(U_j,3\eta'(x')), \\
    linear, & \mbox{if }\,\, t\in [1-\frac{1}{K},1]_i{\cap}\overline{B(U_j,3\eta'(x'))}\backslash B(U_j,2\eta'(x')),\\
    {\rm sgn}(\beta_{ij})\cdot\varphi_1^i(f_{11}^j), & \mbox{if }\,\, t\in [1-\frac{1}{K},1]_i{\cap}\overline{B(U_j,2\eta'(x'))}.
    \end{cases}
   $$
   Then $\chi_j\in A$ and $\phi({\chi_j})$ defines a continuous projection-valued function on an open neighbourhood $B(x',\delta'(x'))$ and denote it by
   ${q_j}|_{B(x',\delta'(x'))}(x)$, then
   $$
   {\rm rank}\,\,{q_j}|_{B(x',\delta'(x'))}(x)\geq \#_{X_j}\big (Sp\phi_x\cap B(U_j,\eta'(x'))\big)\geq \#_{X_j}\big(Sp\phi_{x}\cap U_j\big)
   $$
   for each $x\in B(x',\delta'(x'))$.

   Apply lemma \ref{3.1}, there is a projection-valued function $p_j$ defined on the whole set $[0,1]$ with ${\rm rank}\ p_j=\#(\Lambda_1({h_j}))$, and
   $$
   p_j(x)\leq\bigvee\{{q_j}|_{B(x',\delta'(x'))}(x)\,|\,B(x',\delta'(x')\in\mathcal{W}\}\leq\overline{f}_{11}^j(x).
   $$

   In general, $p_j(x)$ does not belong to $B$, this problem can be solved by using the same technique as in case 2 in step 2, we can obtain a projection new\,$p_j\in B$.

   Set
   $$
   P_j(x)=\sum_{s=1}^{k_j}\overline{f}_{s1}^j(x)\cdot{\rm new}\,p_j(x)\cdot\overline{f}_{1s}^j(x).
   $$
   Now we have $P_j(x)\leq E_{\widetilde{W_j}}(x)$, and $P_{j_1}(x)P_{j_2}(x)=0$, if $j_1\neq j_2$.

   Since
   $$
   \phi'(g)P_j=P_j\phi'(g),
   $$
   then $P_j(x)$ almost commutes with $\phi(g)(x)$,
   \begin{eqnarray*}
     \|P_j\phi(g)-\phi(g)P_j\| &=& \|P_j\phi(g)-P_j\phi'(g)+\phi'(g)P_j-\phi(g)P_j\| \\
      &\leq & \|P_j\phi(g)-P_j\phi'(g)\|+\|\phi'(g)P_j-\phi(g)P_j\|  \\
      &\leq & \frac{2}{J},
   \end{eqnarray*}
   for all $g\in G$.

   From the two cases above, we show that for each $\widetilde{V_j}$, we can find a projection $P_j$ almost commutes with $\phi(g)$, for all $g\in G$ and ${\rm rank}\,\,P_j=k_j\#\Lambda_1(h_j)$ for all $j\in \{1,2,\cdots,p\}$.

   $Step\,4$:
From setp 2 and step 3, we have constructed a collection of mutually orthogonal projections
   $$
   P_1^1,\cdots,P_{K-1}^1,\cdots,P_1^l,\cdots,P_{K-1}^l,P_1,P_2,\cdots,P_p.
   $$
   Each of them almost commutes with $\phi(g)$ for all $g\in G$.

   Define
   $$
   q=1-\sum_{j=1}^{p}P_j-\sum_{i=1}^{l}\sum_{r=1}^{K-1}P_r^i.
   $$

   Change all the spectra in $Sp(\phi_x')\cap W_r^i$ (equal to $Sp(\phi_x)\cap W_r^i$) to $z_r^i$ for each $i=1,2,\cdots,l$, $r=1,2,\cdots,K-1$. We will obtain a pointwise homomorphism $\phi''$  from $\phi'$.

    Define $\psi: A\,\rightarrow (1-q)B(1-q)$ by
   $$
   \psi(f)=(1-q)\phi''(f)(1-q).
   $$
   Then $\psi$ can be factored through a finite dimensional $C^*$-algebra. Now we have
   \begin{eqnarray*}
     \|\phi(g)q-q\phi(g)\|  &=& \|\phi(g)q-q\phi(g)q+q\phi(g)q-q\phi(g)\| \\
      &\leq& \|(1-q)\phi(g)q\|+\|q\phi(g)(1-q)\| \\
      &=& \|\big((1-q)\phi(g)-\phi(g)(1-q)\big)q\|+\|q\big((1-q)\phi(g)-\phi(g)(1-q)\big)\| \\
      &<& \frac{2}{J}+\frac{2}{J}=\frac{4}{J},
   \end{eqnarray*}
   and
   $$
   \|\phi(g)-q\phi(g)q\oplus\psi(g)\|<\frac{4}{J},
   $$
   for any $g\in G$. This means  that conditions (2) and (3) hold.

   It is only to verify condition (1), by the fact that
   $$
   \sum_{r=0}^{K}\#(Sp\phi_{x_0}\cap V_r^i)=\sum_{i=1}^{l}\sum_{r=0}^{K}c_r^il_i \geq\frac{L}{L+1}\omega_i,
   $$
   $$
   {\rm rank}\,P_j\geq k_jt_j+k_j\sum_{i=1}^{l}\alpha_{ij}c_0^i+k_j\sum_{i=1}^{l}\beta_{ij}c_K^i,
   $$
   where $\omega_i=\#(Sp(\phi_{x_0})\cap (0,1)_i)$. Then we have
   \begin{eqnarray*}
     \sum_{j=1}^{p}rankP_j+\sum_{i=1}^{l}\sum_{r=1}^{K-1}rank P_r^i &\geq& \sum_{j=1}^{p}\big(k_jt_j+k_j\sum_{i=1}^{l}\alpha_{ij}c_0^i+k_j\sum_{i=1}^{l}\beta_{ij}c_K^i\big)\\
     & & + \sum_{i=1}^{l}\sum_{r=1}^{K-1}rankP_r^i \\
     &=& \sum_{j=1}^{p}k_jt_j+\sum_{i=1}^{l}c_0^il_i+\sum_{i=1}^{l}c_K^il_i+\sum_{i=1}^{l}\sum_{r=1}^{K-1}c_r^il_i \\
     &\geq& \sum_{j=1}^{p}k_jt_j+\sum_{i=1}^{l}\frac{L}{L+1}\omega_il_i.
   \end{eqnarray*}
   It is obvious that
   $$
   L\cdot{\rm rank}\,\,q<{\rm rank}\,\,(1-q).
   $$
   Moreover, for each projection $e\in A$, we will have
   $$
   L\cdot{\rm rank}\,\,(\phi(e)-\psi(e))<{\rm rank}\,\,\psi(e).
   $$
   Hence, (1) holds.

  \end{proof}
\begin{remark}
  This theorem has some similarity to what was proved in \cite{EG} and \cite{EGS}, the authors use small spectra variation to find finitely many
  projections such that (1), (2), (3) hold. But in general, the spectra may not have a regular form and the projections
  will not belong to $B$, which makes that the small spectra variation may not be sufficient.

  If we assume each Elliott-Thomsen building block has the property that there exist an upper bound for all $p$, then we can take $\eta>0$ small enough and choose a suitable decomposition such that for any projection $e\in A$, we have $$\,L\cdot [\phi(e)-\psi(e)]\leq[\psi(e)].$$
\end{remark}

\begin{corollary}
  Let $A=A(F_1,F_2,\varphi_0,\varphi_1),B=B(F_1',F_2',\varphi_0',\varphi_1')\in \mathcal{C}$, let  $G\subset A$ be a finite set, 
  then for any positive integers $J,L$, there exists $\eta>0$ such that if a homomorphism $\phi: A\rightarrow B$ satisfies $\phi(H(\eta))\subset_{\frac{1}{6}}\{ f\in B\, |f\,has\,finite\,spectrum\}$, then there exists a projection $q\in B$
  and a unital homomorphism $\psi: A\rightarrow (1-q)B(1-q)$ with finite dimensional image such that

  $(1)\,\,L\cdot{\rm rank}\,\,(\phi(e)-\psi(e))<{\rm rank}\,\,\psi(e)$ for any projection $e\in A$,

  $(2)\,\,\|q\phi(g)-\phi(g)q\|<\dfrac{4}{J} \,\,for\, \,any\,\, g\in G$,

  $(3)\,\,\|\phi(g)-q\phi(g)q\oplus\psi(g)\|<\dfrac{4}{J} \,\,for\,\,any\,\, g\in G$.
 \end{corollary}
 \begin{proof}
  we consider $\pi^{i'}\circ \phi$ instead of $\phi$ and repeat the four steps in \ref{dc}.
  Note that in this case $Sp(B)=Sp(F_1')\cup$ $\coprod_{i'=1}^{l'}(0,1)_{i'}$, then for each $i'\in\{1,2,\cdots,l'\}$, we will
  choose $(x_0,i')\in (0,1)_{i'}$ be the base points, then we can prove the result in the same way, so we omit this proof.
  \end{proof}
\begin{corollary}
  Let $A=\underrightarrow{lim}(A_n,\phi_{m,n})$ be a $\C^*$-algebra which is an inductive limit of the Elliott-Thomsen building blocks $A_m$ with morphisms $\phi_{m,n}:A_m\rightarrow A_n$. If RR(A)=0, then for any finite subset $G\subset A_m$, $\varepsilon>0$, $L>0$, there exists $n\geq m$, such that there are two orthogonal projections $q_1,q_2\in A_{m}$ and a unital homomorphism $\phi':A_m\rightarrow q_2A_nq_2$ with finite dimensional image,
  satisfying the following:

  $(1)\,\,{\rm rank}\,\,q_2\geq L\cdot{\rm rank}\,\,q_1,\,\,where\,\,q_1+q_2=\phi_{m,n}(1_{A_m})$,

  $(2)\,\,\|q_1\phi(g)-\phi(g)q_1\|<\epsilon\,\,\,for\,\,any\,\, g\in G,$

  $(3)\,\,\|\phi_{m,n}(g)-q_1\phi_{m,n}(g)q_1\oplus\phi'(g)\|<\epsilon \,\,\,for\,\,any\,\,g\in G.$
\end{corollary}

\noindent{\bf Acknowledgement}

The author would like to thank Guihua Gong for helpful suggestions and Yuanhang Zhang for helpful discussions.

\noindent

Zhichao Liu, Department of Mathematics, Jilin University, 130012, Changchun, P. R. China

E-mail address: lzc.12@outlook.com

\end{document}